\documentclass[11pt, 
a4paper,oneside,centertags,reqno, 
]{amsart}
\usepackage{appendix}
\usepackage{csquotes} 
\usepackage{tikz}
\usetikzlibrary{decorations.pathreplacing}

\usepackage{etoolbox}
\robustify{\footnote}

\usepackage[multiple]{footmisc}

\usepackage{txfonts}
\usepackage{exscale}
\usepackage{amsbsy}
\usepackage[colorlinks=true,citecolor=red,pagebackref,hypertexnames=false]{hyperref}

\usepackage{lmodern}

\usepackage{endnotes}

\usepackage{centernot}

\usepackage{latexsym}
\usepackage{amsmath}
\usepackage{amsfonts}
\usepackage{amssymb}
\usepackage{graphpap}
\usepackage{epsfig}
\usepackage{amscd}
\usepackage{amsthm}
\usepackage{enumerate}
\usepackage{eucal}
\usepackage{esint}

\usepackage{wela} 
\usepackage{calligra} 
\usepackage[utf8]{inputenc}
\usepackage[francais]{[babel}
\usepackage[deutsch]{[babel}
\usepackage[utf8]{inputenc}
\usepackage{graphicx}

\usepackage[backrefs,non-compressed-cites,initials,nobysame,]{amsrefs}

\usepackage{pdfsync}
\usepackage{color}
\definecolor{darkgreen}{rgb}{0.0, 0.5, 0.11}

\usepackage{bbm}
\usepackage{pgf,tikz}
\usepackage{mathrsfs}
\usetikzlibrary{arrows}

\usepackage{pgfplots}
\pgfplotsset{width=10cm,compat=1.9}
\font\got=eufm10 at 11pt
\font\gotm=eufm10 at 8pt
\font\posebni=msam10

\def\Re{\operatorname{Re}}
\def\Im{\operatorname{Im}}

\newcommand{\ca}[0]{\mathbbm{1}}

\newcommand{\C}[0]{{\mathbb C}}

\newcommand{\R}[0]{\mathbb{R}}

\newcommand{\Z}[0]{{\mathbb Z}}

\newcommand{\leqsim}[0]{\,\text{\posebni \char46}\,}

\newcommand{\cA}[0]{{\mathcal A}}

\newcommand{\cM}[0]{{\mathcal M}}

\newcommand{\cV}[0]{{\mathcal V}}

\newcommand{\oA}[0]{{\mathscr A}}

\newcommand{\oE}[0]{{\mathscr E}}
\newcommand{\oL}[0]{{\mathscr L}}
\newcommand{\oM}[0]{{\mathscr M}}

\newcommand{\oV}[0]{{\mathscr V}}

\newcommand{\oH}[0]{{\mathscr H}}

\newcommand{\gota}[0]{{\text{\got a}}}
\newcommand{\gotb}[0]{{\text{\got b}}}

\newcommand{\gotma}[0]{{\text{\gotm a}}}

\newcommand{\gotM}[0]{{\text{\got M}}}

\newcommand\bS{\mathbf{S}}

\newcommand{\mn}[2]{\{ #1 : #2 \}}

\newcommand{\sk}[2]{\left\langle #1 , #2\right\rangle}
\renewcommand{\div}[0]{{\rm div}\,}

\def\Dom{\operatorname{D}}
\def\Ran{\operatorname{R}}

\newtheorem{theorem}{Theorem}
 
\newtheorem{lemma}[theorem]{Lemma}
\newtheorem{proposition}[theorem]{Proposition}
\newtheorem{corollary}[theorem]{Corollary}

\renewcommand\leq[0]{\leqslant}
\renewcommand\geq[0]{\geqslant}
\renewcommand\epsilon[0]{\varepsilon}

\newcommand\wrt{\,\text{\rm d}}
\renewcommand\mod[1]{\left\vert{#1}\right\vert}

\newcommand\norm[2]{{\left\Vert{#1}\right\Vert_{#2}}}

\newtheorem{preremark}[theorem]{Remark}  \newenvironment{remark}%
{\begin{preremark}\rm}{\end{preremark}}

\begin{document}
\title[Maximal operator]{On semigroup maximal operators associated with divergence-form operators with complex coefficients}

\author[A. Carbonaro]{Andrea Carbonaro}
\address{Andrea Carbonaro, Universit\`{a} degli Studi di Genova, Dipartimento di Matematica, Via Dodecaneso 35, 16146 Genova, Italy}
\email{carbonaro@dima.unige.it}

\author[O. Dragi\v{c}evi\'c]{Oliver Dragi\v{c}evi\'c}
\address{Oliver Dragi\v{c}evi\'c, Department of Mathematics, Faculty of Mathematics and Physics, University of Ljubljana,  Jadranska 19, SI-1000 Ljubljana, Slovenia, and Institute of Mathematics, Physics and Mechanics, Jadranska 19, SI-1000 Ljubljana, Slovenia}
\email{oliver.dragicevic@fmf.uni-lj.si}

\subjclass[2020]{42B25, 47D06, 47A60, 35J15} 

\date{November 22, 2022}

\maketitle

\begin{abstract} 
Let $L_{A}=-\div(A\nabla)$ be an elliptic divergence form operator with bounded complex coefficients subject to mixed boundary conditions on an arbitrary open set $\Omega\subseteq\R^{d}$. We prove that the maximal operator $\oM^{A} f=\sup_{t>0}\mod{\exp(-tL_{A})f}$ is bounded in $L^{p}(\Omega)$ whenever $A$ is $p$-elliptic in the sense of \cite{CD-DivForm}. The relevance of this result is that, in general, the semigroup generated by $-L_{A}$ is neither contractive in $L^{\infty}$ nor positive, therefore neither the Hopf--Dunford--Schwartz maximal ergodic theorem \cite[Chap.~VIII]{DunfordSchwartz1} nor Akcoglu's maximal ergodic theorem \cite{Akc75} can be used. We also show that if $d\geq 3$ and the domain of the sesquilinear form associated with $L_{A}$ embeds into $L^{2^{*}}(\Omega)$ with  $2^{*}=2d/(d-2)$, then the range of $L^{p}$-boundedness of $\oM^{A}$ improves to $(rd/((r-1)d+2),rd/(d-2))$, where $r\geq 2$ is such that $A$ is $r$-elliptic. With our method we are also able to study the boundedness of the two-parameter maximal operator $\sup_{s,t>0}\mod{T^{A_{1}}_{s}T^{A_{2}}_{t}f}$.
\end{abstract}

\section{Principal result of the paper}\label{s: M princ}
Let $\Omega\subseteq\R^d$ be an arbitrary open set.
Denote by $\cA(\Omega)$ the family of all complex {\it uniformly strictly accretive} (also called {\it elliptic}) $d\times d$ matrix functions on $\Omega$ with $L^{\infty}$ coefficients.
That is, $\cA(\Omega)$ is the set of all measurable $A:\Omega\rightarrow\C^{d\times d}$ for which
there exist $\lambda,\Lambda>0$ such that for almost all $x\in\Omega$ we have
\begin{eqnarray}
\label{eq: ellipticity}
\Re\sk{A(x)\xi}{\xi}
&\hskip -19pt\geq \lambda|\xi|^2\,,
&\hskip 20pt\forall\xi\in\C^{d};
\\
\label{eq: boundedness}
\mod{\sk{A(x)\xi}{\eta}}
&\hskip-6pt\leq \Lambda \mod{\xi}\mod{\eta}\,,
&\hskip 20pt\forall\xi,\eta\in\C^{d}.
\end{eqnarray}
For any $A\in\cA(\Omega)$ denote by
$\lambda(A)$
the largest admissible $\lambda$ in \eqref{eq: ellipticity} and by
$\Lambda(A)$
the smallest $\Lambda$ in \eqref{eq: boundedness}.

Given $A\in\cA(\Omega)$ and $p\in (1,\infty)$, we say that $A$ is {\it $p$-elliptic}  \cite{CD-DivForm} if
$\Delta_{p}(A)>0$, where
\begin{equation}
\label{eq: kabuto}
\Delta_{p}(A):=
\underset{x\in\Omega}{{\rm ess}\inf}
\min_{\substack{\xi\in\C^{d}\\ |\xi|=1}}
\Re\sk{A(x)\xi}{\xi+|1-2/p|\bar\xi}_{\C^{d}}.
\end{equation}

It follows straight from \eqref{eq: kabuto}
that $\Delta_2(A)=\lambda(A)$, and that $\Delta_{p}$ is invariant under conjugation of $p$, meaning that $\Delta_{p}(A)=\Delta_{p^{\prime}}(A)$, where $1/p+1/p^{\prime}=1$. Denote by $\cA_{p}(\Omega)$ the class of all $p$-elliptic matrix functions on $\Omega$. Then $\cA_{p}(\Omega)=\cA_{p^{\prime}}(\Omega)$. It is known \cite{CD-DivForm} that $A\in\cA_{p}(\Omega)$ if and only $A^{*}\in \cA_{p}(\Omega)$, where $A^{*}$ denotes the conjugate transpose of $A$. Moreover, $\mn{\cA_{p}(\Omega)}{p\in[2,\infty)}$ is a decreasing chain of matrix classes such that
$$
\begin{array}{rcc}
\{\text{elliptic matrices on }\Omega\} & = & \cA_2(\Omega)\,, \\
\{\text{real elliptic matrices on }\Omega\} & = &
{\displaystyle\bigcap_{p\in[2,\infty)}\cA_{p}(\Omega)}\,.
\end{array}
$$
We will refer to $\lambda(A)$, $\Lambda(A)$ and $\Delta_{p}(A)$ collectively as
the $p$-{\rm ellipticity constants} of $A$.

The reader interested in the genesis of $p$-ellipticity and its properties should consult \cite{CD-DivForm, CD-MIXED, Trilinear}. Here we mention that M. Dindo\v{s} and J. Pipher \cite{Dindos-Pipher}, while studying reverse H\"older inequalities for weak solutions of complex elliptic operators, discovered (independently of the authors of the present paper) a condition equivalent to $p$-ellipticity which can be considered a strengthening of \cite[(2.25)]{CiaMaz}. Since then the same authors have continued the study of $p$-ellipticity, applying it successfully in different contexts; see their recent papers \cite{Dindos-Pipher2, DiLiPi20}.

\medskip

Suppose that either
\begin{enumerate}[(a)]
\item
$\oV=W^{1,2}_{0}(\Omega)$,
\item
$\oV=W^{1,2}(\Omega)$, or
\item
$\oV$ is the closure in $W^{1,2}(\Omega)$ of the set of restrictions $\mn{f|_\Omega}{f\in C_c^\infty(\R^d\backslash D)}$, where $D$ is a (possibly empty) closed subset of $\partial\Omega$.
\end{enumerate}
For every $A\in\cA(\Omega)$ we denote by $L_{A}$ the maximal accretive Kato-sectorial operator on $L^{2}(\Omega)$ associated with the densely defined, closed and sectorial form
\begin{equation*}\label{eq: M sesqui}
\gota(f,g):=\int_\Omega\sk{A\nabla f}{\nabla g}_{\C^d},\quad f,g\in\oV.
\end{equation*}
We denote by $(T^{A}_{t})_{t>0}$ the associated contractive analytic semigroup on $L^{2}(\Omega)$; see \cite[Chapter VI]{Kat}, \cite{AuTc} and \cite[Chapters~I~and~IV]{Ouhabook}. It was proven in \cite[Theorem~1.3]{CD-DivForm}, \cite[Theorem~2]{Egert2018} and \cite[Lemma~17]{CD-MIXED} that $(T^{A}_{t})_{t>0}$ extends to a contractive analytic semigroup on $L^{p}(\Omega)$ whenever $A\in\cA_{p}(\Omega)$.

The {\it maximal operator} associated with $(T^{A}_{t})_{t>0}$ is defined by the rule
$$
\oM^{A}f(x):=\sup_{t>0}\mod{T^{A}_{t}f(x)}.
$$
To the best of our knowledge, boundedness of $\oM^{A}$ for general complex elliptic $A$ is unknown. 

The following is our principal result.

\begin{theorem}\label{t: M princ} Let $p\in (1,\infty)$ and $A\in\cA_{p}(\Omega)$. There exists $C>0$, that depends only on the $p$-ellipticity constants of $A$, such that
\begin{equation}\label{eq: M princ}
\norm{\oM^{A}f}{p}\leq C\norm{f}{p},\quad \forall f\in L^{p}(\Omega).
\end{equation}
\end{theorem}
Unless the form $\gota$ is given by real coefficients, the semigroup $(T^{A}_{t})_{t>0}$ is neither bounded in $L^{\infty}(\Omega)$ nor positive \cite[Theorems 4.14, 4.2]{Ouhabook}. The novelty of Theorem~\ref{t: M princ} is that the relevant literature known to us deals with semigroups of sub-positive contraction whose generators have a bounded holomorphic functional calculus in $L^{p}$; see  \cite{cowling} and \cite[Chapter~III]{Stein70}. 

The known methods allow proving \eqref{eq: M princ} only in the case where $A$ is a complex rotation of a real matrix  (see Section~\ref{s: M complex rot}). Namely, \cite{cowling,Stein 70} exploit the Hopf--Dunford--Schwartz maximal ergodic theorem and its generalizations, that is to say, the boundedness of the {\it maximal ergodic operator}
 (defined in \eqref{eq: MaxErgOp} for the case of divergence-form operators)
associated with a positive contractive semigroup  \cite[Chap.~VIII]{DunfordSchwartz1}, \cite[Chap.~4]{CW}, \cite{Akc75} and \cite{Garsia65}; see also \cite[Section~10.7.d]{HNVW2} for a summary of the known results in this direction. However, in the case of complex matrices $A$ there is no general result that replaces the Hopf--Dunford--Schwartz maximal ergodic theorem. Instead, the boundedness of the maximal ergodic operator associated with $(T^{A}_{t})_{t>0}$ will be an immediate consequence of our Theorem~\ref{t: M princ}.

\subsection{Principal ideas for the proof of Theorem~\ref{t: M princ}}\label{s: M main ideas} 
It turns out that for real elliptic matrices $B$ the maximal operator $\oM^{B}$ is bounded in $L^{p}(\Omega)$ for all $p\in (1,\infty]$ (see Corollary~\ref{c: M maximal real}). 
\smallskip

The \framebox{first idea} for proving Theorem~\ref{t: M princ} comes from afar (for instance, see \cite{Stein76}): by triangular inequality,
\begin{equation}\label{eq: M split}
\oM^{A}f(x)\leq\oM^{B}f(x)+\sup_{t>0}\mod{T^{A}_{t}f(x)-T^{B}_{t}f(x)},
\end{equation}
therefore we reduce the problem to proving that the second maximal operator on the right-hand side is bounded in $L^{p}(\Omega)$. The point is that the maximal operator
$$
\oM^{A,B}f(x)=\sup_{t>0}\mod{T^{A}_{t}f(x)-T^{B}_{t}f(x)}
$$
is expected to behave much better than $\oM^{A}$. As an example, consider the case of the two semigroups $\exp(-tzL)$ and $\exp(-tL)$, where $\Re z>0$ and $L=-\Delta$ is the positive Euclidean Laplacian. The Fourier multiplier $m_{z}(\xi):=\exp(-z\mod{\xi}^{2})-\exp(-\mod{\xi}^{2})$ satisfies the estimates $\mod{m_{z}(\xi)}\leqsim \min(|\xi|,|\xi|^{-1})$ and $\mod{\xi\cdot\nabla m_{z}(\xi)}\leqsim 1$, from which one can prove \cite{Bou86,Carbery86,BMSW18} that the square function 
$$
f\mapsto\left(\sum_{n\in\Z}\mod{\exp(-2^{n}zL)f-\exp(-2^{n}L)f}^{2}\right)^{1/2}
$$
and, consequently, the maximal operator 
$$
f\mapsto \sup_{t>0}\mod{\exp(-tL)f-\exp(-tzL)f}
$$ 
are bounded in $L^{p}(\R^{d})$ for all $p\in (1,\infty)$.
\smallskip

The \framebox{second idea} used in the proof of Theorem~\ref{t: M princ} comes from \cite{cowling}, where the author developed a subordination method for studying the maximal operator $\sup_{t>0}\mod{m(t\oA)f}$ associated with, say, a generator $\oA$ of a holomorphic semigroup with bounded imaginary powers on $L^{p}$ and a multiplier $m\in L^{1}(\R_{+},\wrt\lambda/\lambda)$: at least formally we have 
\begin{equation}\label{eq: MM}
m(t\oA)f=\frac{1}{2\pi}\int_{\R}t^{iu}[\cM m](u)\oA^{iu}f\wrt u,
\end{equation}
where $\cM m$ denotes the Mellin transform of $m$, which is defined by rule
\begin{equation}\label{eq: MM1}
[\cM m](u)=\int^{\infty}_{0}m(\lambda)\lambda^{-iu}\frac{\wrt \lambda}{\lambda},\quad u\in\R.
\end{equation}
It follows from \eqref{eq: MM} that
\begin{equation}\label{eq: MM2}
\norm{\sup_{t>0}\mod{m(t\oA)f}}{p}\leq \frac{1}{2\pi}\int_{\R}\mod{[\cM m](u)}\norm{\oA^{iu}f}{p}\wrt u.
\end{equation}
This reduces the estimate of the $L^{p}$-norm of maximal function on the left-hand side of \eqref{eq: MM} to finding poinwise estimates of $\cM m$ and $L^{p}$-estimates of imaginary powers of $\oA$.

If we apply this method to the multiplier

\begin{equation}\label{eq: M multiplier}
m_{\pm\theta}(\lambda):=\exp(-e^{\pm i\theta}\lambda)-\exp(-\lambda),\quad \lambda>0,\quad 0<\theta<\pi/2,
\end{equation}
and the generator $L_{B}$ where $B\in\cA(\Omega)$ is real, then it follows from \cite[Theorem~3]{CD-MIXED} that $\oM^{e^{\pm i\theta} B,B}$ is bounded in $L^{p}(\Omega)$ whenever the matrix $A=e^{\pm i\theta}B$ is $p$-elliptic (see Section~\ref{s: M complex rot} for details). This argument breaks down when the complex rotations of $B$ are replaced by a general $A\in\cA_{p}(\Omega)$, because the two generators $L_{A}$ and $L_{B}$, in general, are not related via functional calculus (see also the comments in Section~\ref{s: M proof}). Let us explain how we get around this.
\medskip

Let $A,B\in\cA(\Omega)$, where $B$ is real. We have the Duhamel's formula:
\begin{equation}\label{eq: M Duhamel1}
T^{A}_{t}f-T^{B}_{t}f=\int^{t}_{0}T^{B}_{s}(L_{B}-L_{A})T^{A}_{t-s}f\wrt s,\quad \forall f\in (L^{2}\cap L^{p})(\Omega).
\end{equation}
Duhamel's formula is typically used for showing that the difference of the two semigroups is small (with respect to some topology) whenever $A-B$ is small, for example, with respect to the $L^{\infty}$-norm. So it would seem futile to use the estimate
\begin{equation}\label{eq: M Duhamel}
\mod{T^{A}_{t}f-T^{B}_{t}f}\leq \mod{\int^{t}_{0}L_{B}T^{B}_{s}T^{A}_{t-s}f\wrt s}+\mod{\int^{t}_{0}T^{B}_{s}L_{A}T^{A}_{t-s}f\wrt s},
\end{equation}
but actually it works. 

For simplicity consider the first term on the right-hand side \eqref{eq: M Duhamel}. Our idea is to transfer a fractional power of $L_{B}$ from the first semigroup to the second one as follows: 
$$
L_{B}T^{B}_{s}T^{A}_{t-s}=L^{1-\alpha}_{B}T^{B}_{s}L^{\alpha}_{B}T^{A}_{t-s},\quad \alpha\in (0,1).
$$
Now we would like to transform the fractional power of $L_{B}$ into a fractional power of $L_A$, more precisely we aim at the decomposition 
$$
L^{\alpha}_{B}=U^{B,A}_{\alpha}L^{\alpha}_{A},
$$
where $U^{B,A}_{\alpha}$ is a bounded operator on $L^{p}(\Omega)$.
Yet this is not possible for all $A, B, p,\alpha$, because for, say, $1/2\leq \alpha<1$ and $p=2$, the domains of $L^{\alpha}_{A}$ and $L^{\alpha}_{B}$ are either unrelated or unknown. This procedure with $B=I$, $\alpha=1/2$ and $p=2$ would require a solution to the Kato problem for the square root of $L_{A}$ which is still an open problem in the general setting we consider in this paper. 

However, by combining a classic result of Kato \cite{Kato61} with \cite[Theorem~3]{CD-MIXED} and a classic complex interpolation argument, we are able to make the procedure described above work for a small positive power $\alpha=\alpha(p)<1/2$ whenever $A\in\cA_{p}(\Omega)$ (see Proposition~\ref{p: M frac transfer}). This gives
$$
\int^{t}_{0}L_{B}T^{B}_{s}T^{A}_{t-s}f\wrt s=\int^{t}_{0}\psi_{1-\alpha}(sL_{B})U^{B,A}_{\alpha}\psi_{\alpha}((t-s)L_{A})f\wrt\mu^{t}_{\alpha}(s)
$$
with $U^{B,A}_{\alpha}$ as above,
\begin{equation}\label{eq: M mu}
\wrt\mu ^{t}_{\alpha}(s)=s^{\alpha-1}(t-s)^{-\alpha}\ca_{[0,t]}(s)\wrt s,\quad t>0,
\end{equation}
and
\begin{equation}\label{eq: M psi}
\psi_{\beta}(\lambda)=\lambda^{\beta}e^{-\lambda},\quad \lambda>0,\quad \beta>0.
\end{equation}
The advantage of this decomposition is that $\mu^{t}_{\alpha}(\R)$ is a finite measure on $\R$ with total mass independent of $t$. More precisely,
\begin{equation}\label{eq: M psi bis}
\mu^{t}_{\alpha}(\R)=B(\alpha,1-\alpha)=\frac{\pi}{\sin(\alpha\pi)},\quad \forall t>0.
\end{equation}
Then, by using Cowling's subordination \cite{cowling} twice in combination with our functional calculus result in \cite[Proposition~3]{CD-MIXED}, we obtain that the two-variable maximal operator
$$
\sup_{0<s\leq t}\mod{\psi_{1-\alpha}(sL_{B})U^{B,A}_{\alpha}\psi_{\alpha}((t-s)L_{A})f}
$$
is bounded in $L^{p}(\Omega)$, whenever $A\in\cA_{p}(\Omega)$ (see Section~\ref{s: M proof of THM}). In this way we can prove that $\oM^{A,B}$ --- and thus $\oM^{A}$ --- is bounded in $L^{p}(\Omega)$ whenever $A\in\cA_{p}(\Omega)$.

\section{More notation and preliminaries}\label{s: prelim}

Let $1<p<2$ and $A\in\cA_{p}(\Omega)$. With a slight abuse of notation, we maintain the symbol $L_{A}$ for denoting the negative generator of $(T^{A}_{t})_{t>0}$ in $L^{q}(\Omega)$, $p\leq q\leq p^{\prime}$. We have $L^{q}(\Omega)={\rm N}_{q}(L_{A})\oplus \overline{\Ran_{q}(L_{A})}$ and $Q_{q}f=\lim_{t\rightarrow+\infty}T^{A}_{t}f$ is a consistent family of contractive projections onto ${\rm N}_{q}(L_{A})$ for $p\leq q\leq p^{\prime}$ \cite{CDMY}. Hence, ${\rm N}_{q}(L_{A})=\{0\}$ if and only if ${\rm N}_{2}(L_{A})=\{0\}$. A simple calculation based on the definition of $L_{A}$ by means of the sesquilinear form $\gota$ shows that ${\rm N}_{2}(L_{A})={\rm N}_{2}(L_{I})$ consists of the vector space of all $L^2$ functions which are constant on each connected component of $\Omega$ of finite measure whose boundary does not intersect $D$ in a set of positive relative $2$-capacity \cite{AreWarm03}, and are $0$ on all other connected components of $\Omega$. In order to simplify our proofs, in this paper we shall always assume that  ${\rm N}_{2}(L_{I})=\{0\}$.

Under this assumption $L_{A}$ is an injective sectorial operator of angle $<\pi/2$ on $L^{q}(\Omega)$, whenever $p\leq q\leq p^{\prime}$. In this range of $q$'s the complex powers $L^{z}_{A}$ are well defined on $L^{q}(\Omega)$ for every $z\in\C$. Also, for every $z\in\C$, the closed densely defined operators $L^{z}_{A}$ on $L^{q}(\Omega)$, $p\leq q\leq p^{\prime}$, are consistent \cite{Haase, Yagi2010}. 
\begin{proposition}[\cite{CD-MIXED}]\label{p: M Mixed}
Let $p>1$ and $A\in\cA_{p}(\Omega)$. Then $L_{A}$ has bounded $H^{\infty}$-calculus of angle $<\pi/2$ on $L^{p}(\Omega)$ in the sense of \cite{CDMY}. It particular, $L_{A}$ has bounded imaginary powers on $L^{p}(\Omega)$ and there exist $\theta_{p}\in (0,\pi/2)$ and $C_p>0$ such that
\begin{equation}\label{eq: M impowers}
\norm{L^{iu}_{A}}{p}\leq C_pe^{\theta_{p}|u|},\quad \forall u\in\R.
\end{equation}
The angle $\theta_{p}$ and the constant $C_p$ only depend on $p$ and the $p$-ellipticity constants of $A$ \cite[Corollary~5.17 and Proposition~5.23]{CD-DivForm}.
\end{proposition}
\subsection{Cowling's subordination technique}\label{s: M Cowling} It is by now a classic result in semigroup theory that for a semigroup $(S_{t})_{t>0}$ of linear transformations on some  Lebesgue spaces $L^{p}$, the boundedness of the associated maximal ergodic operator in $L^{p}$ together with a $H^{\infty}$-calculus of angle $<\pi/2$ in $L^{p}$ for the negative generator implies the boundedness of the associated maximal operator in $L^{p}$. One way to prove this result is to use Cowling's subordination technique \cite{cowling} (see also \cite{Meda,CD-mult}) which we briefly describe below in the specific case we are interested in here (see \eqref{eq: M comparison}, Proposition~\ref{p: M maximal domination} and Corollary~\ref{c: M maximal real}). Note that Cowling originally formulated his argument only for symmetric semigroups. Note also that as an alternative to Cowling's subordination method, we can slightly modify an argument of Stein \cite[Chapter~III]{Stein70} that exploits complex interpolation and fractional integration.
\medskip

Fix $A\in\cA(\Omega)$. Define the associated {\it maximal ergodic operator} by the rule 
\begin{equation}
\label{eq: MaxErgOp}
\oE^{A}f(x)=\sup_{t>0}\mod{\frac{1}{t}\int^{t}_{0}T^{A}_{s}f(x)\wrt s},\quad f\in L^{2}(\Omega).
\end{equation}
A rapid calculation shows that for all $f\in (L^{2}\cap L^{p})(\Omega)$ and $t>0$ we have
$$
T^{A}_{t}f=\frac{1}{t}\int^{t}_{0}T^{A}_{s}f\wrt s-\frac{1}{t}\int^{t}_{0}\psi_{1}(sL_{A})f\wrt s,
$$
where $\psi_{\alpha}$, $\alpha>0$, is defined by \eqref{eq: M psi} and the bounded operator $\psi_{1}(tL_{A})$ on $L^{2}(\Omega)$ is defined by means of McIntosh's $L^{2}$-functional calculus \cite{Mc}.
Therefore,
\begin{equation}\label{eq: M comparison}
\sup_{t>0}\mod{T^{A}_{t}f(x)}\leq \oE^{A}f(x)+\sup_{t>0}\mod{\psi_{1}(tL_{A})f(x)},\quad \forall f\in (L^{2}\cap L^{p})(\Omega).
\end{equation}

\begin{proposition}\label{p: M maximal domination}
Let $p>1$. Suppose that $A\in\cA_{p}(\Omega)$. For every $\alpha>0$ there exists $C>0$, which depends only on $\alpha$, $p$ and the $p$-ellipticity constants of $A$, such that
$$
\norm{\sup_{t>0}\mod{\psi_{\alpha}(tL_{A})f}}{p}\leq C\norm{f}{p},\quad \forall f\in L^{p}(\Omega).
$$
\end{proposition}
\begin{proof} 
We use Cowling's subordination \eqref{eq: MM2}. The Mellin transform \eqref{eq: MM1} of $\psi_{\alpha}$ is 
\begin{equation}
\label{eq: MT_psi}
[\cM\psi_{\alpha}](u)=\Gamma(\alpha-iu),
\hskip 30pt
\text{ for all }u\in\R, 
\end{equation}
where $\Gamma$ denotes the Euler gamma function. Stirling's formula gives
\begin{equation}\label{eq: M Stirling}
\mod{\Gamma(\alpha-iu)} \sim\sqrt{2\pi}\,|u|^{\alpha-1/2}e^{-\frac{\pi}{2}|u|},\quad \textrm{for } |u|\rightarrow \infty.
\end{equation}
Therefore, 
$$
\norm{\sup_{t>0}\mod{\psi_{\alpha}(tL_{A})f}}{p}\leqsim \int_{\R}(1+|u|)^{\alpha-1/2}e^{-\frac{\pi}{2}|u|}\norm{L^{iu}_{A}f}{p}\wrt u.
$$
Now the desired result follows from Proposition~\ref{p: M Mixed}.
\end{proof}

\begin{corollary}\label{c: M maximal real} Let $B\in\cA(\Omega)$ be a real matrix. Then $\oM^{B}$ is bounded in $L^{p}(\Omega)$, for every $p\in (1,+\infty]$.
\end{corollary}
\begin{proof}
As remarked before, when $B$ is real elliptic, $(T^{B}_{t})_{t>0}$ is positive and contractive on $L^{\infty}(\Omega)$ and on $L^{1}(\Omega)$ by the duality $(T^{B}_{t})^{*}=T^{B^{*}}_{t}$ \cite{Ouhabook}. Hence, by the Hopf--Dunford--Schwartz theorem, $\oE^{B}$ is bounded in $L^{p}(\Omega)$ for all $p>1$. Moreover, $B\in\cA_{p}(\Omega)$ for all $p>1$ \cite[p. 3179]{CD-DivForm}, therefore by Proposition~\ref{p: M maximal domination} the maximal operator $f\mapsto\sup_{t>0}\mod{\psi_{1}(tL_{B})f}$ is bounded in $L^{p}(\Omega)$ for all $p>1$. The desired result for $p\neq \infty$ now follows from \eqref{eq: M comparison}, while the boundedness of the maximal operator in $L^{\infty}(\Omega)$ trivially follows from the contractivity of $(T^{B}_{t})_{t>0}$ on $L^{\infty}(\Omega)$.
\end{proof}

\subsection{Complex rotations of real elliptic matrices}\label{s: M complex rot} In the case of a {\it symmetric contraction semigroup} $(S_{t})_{t>0}$ on some $\sigma$-finite measure space, Cowling \cite{cowling} observed that the subordination technique explained in Section~\ref{s: M Cowling} can be used for proving the boundedness on $L^{p}(\Omega)$, $1<p<\infty$, of the maximal operator 
$$
f\mapsto \underset{|\arg(z)|<\omega_{p}-\epsilon}{\sup}\mod{S_{z}f},
$$
where $\omega_{p}=\pi/2-\omega^{*}_{p}$ and $\omega^{*}_{p}$ denotes the angle of the $H^{\infty}$-calculus in $L^{p}$ of the negative generator of $(S_{t})_{t>0}$. Recall that by \cite{CD-mult}, for every symmetric contraction semigroup we have $\omega^{*}_{p}\leq \phi^{*}_{p}:=\arcsin\mod{1-2/p}$. Despite the fact that for a {\bf real} elliptic matrix $B$, in general, the semigroup $(T^{B}_{t})_{t>0}$ is not symmetric, Cowling's argument can be easily adapted and in combination with Proposition~\ref{p: M Mixed} gives the following extension of Corollary~\ref{c: M maximal real}.
\begin{proposition}\label{p: M complex rot}
Let $B\in\cA(\Omega)$ be a real matrix. Suppose that $1<p<\infty$ and $0\leq \theta<\pi/2$ are such that $e^{i\theta}B\in\cA_{p}(\Omega)$. Then $\oM^{e^{\pm i\theta}B}$ is bounded on $L^{p}(\Omega)$.
\end{proposition}
\begin{proof}

Recall the definition \eqref{eq: M multiplier} of $m_{\pm\theta}$. Then,
$$
\oM^{e^{\pm i\theta}B}f\leq \oM^{B}f+\sup_{t>0}\mod{m_{\pm\theta}(tL_{B})f}.
$$
By Corollary~\ref{c: M maximal real} the maximal operator $\oM^{B}$ is bounded on $L^{p}(\Omega)$. 

As for the maximal operator associated with $m_{\pm\theta}$, we have
$$
[\cM m_{\pm \theta}](u)=i\theta\left(\frac{e^{\mp\theta\,u}-1}{\theta u}\right)\Gamma(1-iu),\quad u\in\R.
$$
Hence by \eqref{eq: M Stirling},
$$
\mod{[\cM m_{\pm \theta}](u)}\leqsim (1+|u|)^{1/2}e^{(\theta-\pi/2)|u|},\quad \forall u\in\R.
$$
Now use the assumptions, the fact that $\cA_{p}(\Omega)$ is invariant under conjugation \cite[Corollary 5.17]{CD-DivForm}, the identity $L_{e^{\pm i\theta}B}=e^{\pm i\theta}L_{B}$, Cowling's subordination \eqref{eq: MM2} applied with $\oA=L_{B}$ and $m=m_{\pm\theta}$, and Proposition~\ref{p: M Mixed} applied with $A=e^{\pm i\theta}B$.
\end{proof}
\begin{remark}
Let $B$ a real elliptic matrix. By \cite[Propositions~5.13~and~5.21]{CD-DivForm} $e^{\pm i\theta}B$ is $p$-elliptic precisely when $\theta<\pi/2-\vartheta^{*}_{p}$, where $\vartheta^{*}_{p}=\vartheta^{*}_{p}(B)$ is given by \cite[(3.7)]{CD-OU}. This is consistent with the general result for symmetric contractions \cite{CD-mult} explained above, because $T^{e^{\pm i\theta}B}_{t}=T^{B}_{te^{\pm i\theta}}$ and when the real matrix $B$ is symmetric we have $\vartheta^{*}_{p}=\phi^{*}_{p}=\arcsin\mod{1-2/p}$.
\end{remark}

\section{Proof of Theorem~\ref{t: M princ}}\label{s: M proof} The proof of Proposition~\ref{p: M complex rot} is based on the fact that $T^{e^{\pm i\theta}B}_{t}-T^{B}_{t}$ coincides with the multiplier $m_{\pm \theta}(tL_{B})$ of the generator $L_{B}$: this is clearly false when we replace the complex rotations of the real matrix $B$ by a general complex $p$-elliptic matrix $A$, since $T^{A}_{t}$ and $T^{B}_{t}$ may not commute and the generators $L_{A}$ and $L_{B}$ may not have a joint functional calculus. As we already explained in Section~\ref{s: M main ideas}, despite this difficulty, we prove Theorem~\ref{t: M princ}
by combining Corollary~\ref{c: M maximal real} with Proposition~\ref{p: M Mixed} and using Cowling's subordination \eqref{eq: MM} appropriately. To do that, another ingredient is needed: a classic result of Kato \cite{Kato61} concerning the domain of the fractional powers of Kato-sectorial operators, which we state below in the particular case we will use in this paper.  
\medskip

For every $\alpha\in [0,1/2]$ denote by $$\oV_{\alpha}=\left[L^{2}(\Omega),\oV\right]_{2\alpha}$$ the complex interpolation space of index $2\alpha$ between $L^{2}(\Omega)$ and the domain $\oV$ of the sesquilinear form associated with $L_{A}$, $A\in\cA(\Omega)$.
\begin{lemma}[Kato~\cite{Kato61} and Lions~\cite{Lions62}]\label{l: M Kato61}
Let $A,B\in\cA(\Omega)$ and $0<\alpha<1/2$. Then
$$
\Dom_{2}(L^{\alpha}_{A})=\Dom_{2}(L^{\alpha}_{B})=\oV_{\alpha},
$$
and there exist $C_{\alpha},C^{\prime}_{\alpha}>0$ depending only on $\alpha$ and the ellipticity constants of $A$ and $B$ such that
\begin{equation}\label{eq: M fract comp}
C_{\alpha}\norm{L^{\alpha}_{A}f}{2}\leq \norm{L^{\alpha}_{B}f}{2}\leq C^{\prime}_{\alpha} \norm{L^{\alpha}_{A}f}{2},\quad \forall f\in\oV_{\alpha}.
\end{equation}
Moreover, for every $\alpha_{0}\in [0,1/2)$ we have $\sup_{0\leq \alpha\leq \alpha_{0}}\left(\mod{C_{\alpha}}^{-1}+\mod{C^{\prime}_{\alpha}}\right)<+\infty$.
\end{lemma}

Recall that for every $A\in \cA(\Omega)$ we have ${\rm N}_{2}(L_{A})={\rm N}_{2}(L_{I})$ and that for simplicity in this paper we always assume that $\Omega$ and $\oV$ are such that ${\rm N}_{2}(L_{I})=\{0\}$. Therefore $\overline{{\rm R}}_{2}(L_{A})=L^{2}(\Omega)$ and \eqref{eq: M fract comp} is equivalent to the fact that
$$
U^{B,A}_{\alpha}:=L^{\alpha}_{B}L^{-\alpha}_{A}
$$
extends to a bounded (invertible) operator in $L^{2}(\Omega)$, for all $\alpha\in (0,1/2)$ and all $A,B\in\cA(\Omega)$. This together with a standard complex interpolation argument based on \eqref{eq: M impowers} gives the next proposition which extends the $L^{2}$-boundedness of $U^{B,A}_{\alpha}$ to the $L^{p}$-boundedness in the range of $p$-ellipticity of $A$ and $B$. 

\begin{proposition}\label{p: M frac transfer}
Let $p>1$. Suppose that $A,B\in\cA_{p}(\Omega)$. Then there exists $\alpha=\alpha(p)\in (0,1/2)$ such that $U^{B,A}_{\alpha}$ extends to a bounded operator in $L^{p}(\Omega)$.
\end{proposition}
\begin{proof}
For simplicity, suppose that $1<p<2$. For every $\alpha>0$ define the strip
$$
\Sigma_{\alpha}=\{z\in\C:\ 0<\Re z<\alpha\}.
$$
Fix $\alpha_{0}\in (0,1/2)$ and consider the family of operators 
$$
S_{z}=e^{z^{2}}L^{i\Im z}_{B}U^{B,A}_{\Re z}L^{-i\Im z}_{A},\quad z\in\Sigma_{\alpha_{0}}.
$$
The imaginary powers of $L_{A}$ and $L_{B}$ are bounded in $L^{2}(\Omega)$ (this is a consequence of \cite{CrDe}, but it also follows from \eqref{eq: M impowers} applied with $p=2$) and, by Lemma~\ref{l: M Kato61}, the operators $U^{B,A}_{\Re z}$ are bounded in $L^{2}(\Omega)$ uniformly in $z\in\overline{\Sigma}_{\alpha_{0}}$. Hence,
\begin{equation}\label{eq: M auxi}
\norm{S_{z}}{2}\leq C(\alpha_{0}),\quad \forall z\in\overline{\Sigma}_{\alpha_{0}}.
\end{equation}
Moreover, by \eqref{eq: M impowers} we have
$$
\norm{S_{i\sigma}}{p}\leq C_{p},\quad \forall \sigma\in\R.
$$
Suppose we have already proved that for all $f\in L^{2}(\Omega)$ the function 
$$
z\mapsto S_{z}f\in L^{2}(\Omega)
$$
is continuous and bounded on $\overline{\Sigma}_{\alpha_{0}}$, and holomorphic in the interior of the strip. Then, by Stein's complex interpolation theorem, for every $q\in (p,2)$ there exist $\alpha(q)\in (0,\alpha_{0})$ and $C_{q}>0$ such that
$$
\norm{U^{B,A}_{\alpha(q)}}{q}\leq C_{q}.
$$
This proves the proposition, because  the interval of $p$-ellipticity of $A$ and $B$ is open \cite{CD-DivForm}.
\medskip

It remains to show that $z\mapsto S_{z}f\in L^{2}(\Omega)$ is continuous and bounded on $\overline{\Sigma}_{\alpha_{0}}$, and holomorphic in the interior of the strip. We first show that
\begin{equation}\label{eq: M fract1}
S_{z}f=e^{z^{2}}L^{z}_{B}L^{-z}_{A}f,\quad \forall z\in\Sigma_{1/2},\quad \forall f\in\Ran_{2}(L_{A}).
\end{equation}
Indeed, since the imaginary powers of  $L_{B}$ are bounded in $L^{2}(\Omega)$, we have $\Dom(L^{z}_{B})=\Dom(L^{\Re z}_{B})$ and $L^{z}_{B}=L^{i\Im z}_{B}L^{\Re z}_{B}$; see, for example, \cite[Proposition~3.2.1 (c)]{Haase} applied with $\alpha=i\Im z$ and $\beta=\Re z$. Therefore, by Lemma~\ref{l: M Kato61} we have
\begin{equation}\label{eq: M fract2}
L^{z}_{B}=L^{i\Im z}_{B}L^{\Re z}_{B}L^{-\Re z}_{A}L^{\Re z}_{A}=L^{i\Im z}_{B}U^{B,A}_{\Re z}L^{\Re z}_{A}.
\end{equation}

Using once again \cite[Proposition~3.2.1 (c)]{Haase} we deduce that
\begin{equation}\label{eq: M fract3}
L^{\Re z}_{A}L^{-z}_{A}\subset \overline{L^{\Re z}_{A}L^{-z}_{A}}=L^{\Re z-z}_{A}=L^{-i\Im z}_{A}.
\end{equation}
By \cite[Proposition~3.1.1 (c)]{Haase} we have $L_{A}=L^{z+(1-z)}_{A}=L^{1-z}_{A}L^{z}_{A}=L^{z}_{A}L^{1-z}_{A}$, which implies
\begin{equation}\label{eq: M fract4}
\Ran_{2}(L_{A})\subseteq \Ran_{2}(L^{z}_{A}),\quad \forall z\in \overline{\Sigma}_{1}.
\end{equation}
Now \eqref{eq: M fract1} follows by combining  \eqref{eq: M fract2} with \eqref{eq: M fract3} and \eqref{eq: M fract4}.

We are now ready to study the regularity of $z\mapsto S_{z}f$. Fix $f\in \Dom_{2}(L_{A})\cap\Ran_{2}(L_{A})$ and $g\in\Dom_{2}(L_{B^{*}})\cap \Ran_{2}(L_{B^{*}})$. Let $z\in\C$ be such that $0\leq\Re z<1/2$. By Lemma~\ref{l: M Kato61} we have $\Dom_{2}(L^{z}_{A})=\Dom_{2}(L^{z}_{B})$, which together with \eqref{eq: M fract4} gives
$$
\sk{L^{z}_{B}L^{-z}_{A}f}{g}_{L^{2}}=\sk{L^{-z}_{A}f}{L^{\Bar z}_{B^{*}}g}_{L^{2}}.
$$
By \cite[Proposition~3.2.1 (f)]{Haase} (applied with $\alpha_{0}=\alpha_{1}=1/2$) the maps $z\mapsto L^{-z}_{A}f$ and $z\mapsto \overline{L^{\Bar z}_{B^{*}}g}$ are analytic in the strip $-1/2<\Re z<1/2$.

We now observe that, by assumption, $L^{2}(\Omega)=\overline{\Ran}_{2}(L_{B^{*}})$. Hence $\Dom_{2}(L_{B^{*}})\cap \Ran_{2}(L_{B^{*}})$ is dense in $L^{2}(\Omega)$. In order to see this, consider $g_{n}\in\Dom(L_{B^{*}})$ such that $L_{B^{*}}g_{n}\rightarrow f$ and set $f_{n}=n\int^{1/n}_{0}T^{B^{*}}_{s}L_{B^{*}}g_{n}\wrt s$. It follows form \eqref{eq: M fract1} and the previous considerations that for every $f\in \Dom_{2}(L_{A})\cap\Ran_{2}(L_{A})$ the map $z\mapsto S_{z}f\in L^{2}(\Omega)$ is continuous in $\overline{\Sigma}_{\alpha_{0}}$ and holomorphic in the interior of the strip. By the uniform boundedness of the family $\{S_{z}: z\in\overline{\Sigma}_{\alpha_{0}}\}$ (see \eqref{eq: M auxi}) and the density of $\Dom_{2}(L_{A})\cap\Ran_{2}(L_{A})$ in $L^{2}(\Omega)$, we deduce that the same is true for every $f\in L^{2}(\Omega)$.
\end{proof}
\subsection{Proof of Theorem~\ref{t: M princ}}\label{s: M proof of THM} Fix a real elliptic matrix $B$ (for example, $B=I$). By Corollary~\ref{c: M maximal real}, the maximal operator associated with $(T^{B}_{t})_{t>0}$ is bounded in $L^{p}(\Omega)$. Hence by \eqref{eq: M split} it suffices to prove that the maximal operator $\oM^{A,B}$ is bounded in $L^{p}(\Omega)$. Fix $f\in (L^{2}\cap L^{p})(\Omega)$. By Duhamel's formula \eqref{eq: M Duhamel} for every $t>0$ we have
$$
\aligned
\mod{T^{A}_{t}f-T^{B}_{t}f}&\leq \mod{\int^{t}_{0}L_{B}T^{B}_{s}T^{A}_{t-s}f\wrt s}+\mod{\int^{t}_{0}T^{B}_{s}L_{A}T^{A}_{t-s}f\wrt s}\\
&=\mod{I_{t}(f)}+\mod{II_{t}(f)}.
\endaligned
$$
\subsubsection{The maximal operator $\sup_{t>0}\mod{I_{t}(f)}$}\label{s: M mm1} Let $\alpha=\alpha(p)\in (0,1/2)$ be as in Proposition~\ref{p: M frac transfer}. Set $\beta=1-\alpha$. Then we have 
\begin{equation}\label{eq: M power-transfer}
I_{t}(f)=\int^{t}_{0}\psi_{\beta}(sL_{B})U^{B,A}_{\alpha}\psi_{\alpha}((t-s)L_{A})f\wrt\mu^{t}_{\alpha}(s),
\end{equation}
where the finite measures $\mu^{t}_{\alpha}$, $t>0$,  are defined by \eqref{eq: M mu}. Hence by \eqref{eq: M psi bis} we have
$$
\sup_{t}\mod{I_{t}(f)}\leq \frac{\pi}{\sin(\alpha\pi)}\sup_{0<s\leq t}\mod{\psi_{\beta}(sL_{B})U^{B,A}_{\alpha}\psi_{\alpha}((t-s)L_{A})f}.
$$
We now subordinate the operators $\psi_{\beta}(sL_{B})$ and $\psi_{\alpha}((t-s)L_{A})$ to imaginary powers by means of \eqref{eq: MM}, and by using \eqref{eq: MT_psi} we obtain
$$
\sup_{0<s\leq t}\mod{\psi_{\beta}(sL_{B})U^{B,A}_{\alpha}\psi_{\alpha}((t-s)L_{A})f}\leq \frac{1}{\pi^{2}} \int_{\R^{2}}\mod{\Gamma(\beta-iu)}\mod{\Gamma(\alpha-iv)}\mod{L^{iu}_{B}U^{B,A}_{\alpha}L^{i v}_{A}f}\wrt u\wrt v.
$$
The boundedness of $f\mapsto \sup_{t}\mod{I_{t}(f)}$ now follows by combining the Stirling's formula \eqref{eq: M Stirling} with Propositions~\ref{p: M frac transfer} and~\ref{p: M Mixed}.

\subsubsection{The maximal operator $\sup_{t>0}\mod{II_{t}(f)}$}\label{s: M mm2}
For every $\alpha\in (0,1/2)$ set 
$$
V_{\alpha}^{B,A}=\left(U^{A^{*},B^{*}}_{\alpha}\right)^{*}.
$$
A rapid calculation shows that
$$
V_{\alpha}^{B,A}\oV_{\alpha}\subseteq\oV_{\alpha}\quad \textrm{and}\quad L^{\alpha}_{B}V_{\alpha}^{B,A}f=L^{\alpha}_{A}f,\quad \forall f\in \oV_{\alpha}.
$$
Fix $\alpha=\alpha(p^{\prime})$ as in Proposition~\ref{p: M frac transfer}, but for the pair $(A^{*},B^{*})$ instead of the pair $(B,A)$. Set $\beta=1-\alpha$. Then $V_{\alpha}^{B,A}$ extends to a bounded operator on $L^{p}(\Omega)$ and
$$
II_{t}(f)=\int^{t}_{0}L^{\alpha}_{B}T^{B}_{s}V_{\alpha}^{B,A}L^{\beta}_{A}T^{A}_{t-s}f\wrt s=\int^{t}_{0}\psi_{\alpha}(sL_{B})V^{B,A}_{\alpha}\psi_{\beta}((t-s)L_{A})f\wrt\mu^{t}_{\beta}(s).
$$
Proceeding as we did in the previous section, we obtain the boundedness on $L^{p}(\Omega)$ of the maximal operator $f\mapsto \sup_{t>0}\mod{II_{t}(f)}$ that completes the proof of Theorem~\ref{t: M princ}.

\section{Range improvement in Theorem~\ref{t: M princ}}
In this section we consider an open set $\Omega\subseteq\R^{d}$, $d\geq 3$. We fix a subspace $\oV$ of the type described on page~\pageref{eq: M sesqui} and we assume that it satisfies the homogeneous embedding property:
\begin{equation}\label{eq: M embedding}
\norm{u}{2^{*}}\leqsim \norm{\nabla u}{2},\quad \forall u\in \oV,
\end{equation}
where $2^{*}=2d/(d-2)$ is the usual Sobolev upper exponent. For every $p\geq 2$ we define the two exponents
$$
p^{\circ}=\frac{pd}{d-2},\quad p_{\circ}=\frac{p^{\circ}}{p^{\circ}-1}.
$$
Note that $2^{\circ}=2^{*}$.
\begin{lemma}[\cite{Egert2018} and \cite{Elst2019}]\label{l: M Egert}
Assume that $\oV$ satisfies \eqref{eq: M embedding}. Let $p>2$ and $A\in\cA_{p}(\Omega)$. Then for all $r\in[p_{\circ},p^{\circ}]$ the semigroup $(T^{A}_{t})_{t>0}$ extends to a uniformly bounded analytic semigroup in $L^{r}(\Omega)$. Moreover, the analyticity angle does not depend on $r\in[p_{\circ},p^{\circ}]$.
\end{lemma}
\begin{remark}
The homogeneous Sobolev embedding \eqref{eq: M embedding} implies ${\rm N}_{2}(L_{I})=\{0\}$ and 
the very same considerations of Section~\ref{s: prelim} show that, under the assumptions of Lemma~\ref{l: M Egert}, we have
$$
\{0\}={\rm N}_{2}(L_{I})={\rm N}_{2}(L_{A})={\rm N}_{r}(L_{A}),\quad \forall r\in[p_{\circ},p^{\circ}].
$$
\end{remark}
In the next result $\vartheta_{2}$ denotes the sectoriality angle of $L_{A}$ in $L^{2}(\Omega)$. We always have
$0\leq\vartheta_{2}\leq \arccos(\lambda/\Lambda)<\pi/2$.
\begin{lemma}[\cite{Egert2018, Egert2}]\label{l: M Egert 2}
Under the assumptions of Lemma~\ref{l: M Egert}, for every $r\in[p_{\circ},p^{\circ}]$ the generator $L_{A}$ has a bounded $H^{\infty}(\bS_{\vartheta})$-calculus in $L^{r}(\Omega)$, whenever $\vartheta>\vartheta_{2}$. In particular, there exists $\vartheta_{0}\in (0,\pi/2)$ such that
\begin{equation}\label{eq: M impower sobolev}
\norm{L^{iu}_{A}}{r}\leq C(r,d)e^{\vartheta_{0}|u|}, \quad \forall u\in\R,
\end{equation}
for all $r\in[p_{\circ},p^{\circ}]$.
\end{lemma}
\begin{theorem}
Assume that $\oV$ satisfies \eqref{eq: M embedding}. Let $p>2$ and $A\in\cA_{p}(\Omega)$. Then the maximal operator $\oM^{A}$ is bounded in $L^{r}(\Omega)$ for all $r\in [p_{\circ},p^{\circ}]$. 
\end{theorem}
\begin{proof}
Exactly as we did in the proof of Theorem~\ref{t: M princ}, we fix a real matrix $B$ and reduce the proof to showing that the maximal operator $\oM^{A,B}$ is bounded in $L^{r}(\Omega)$. By using \eqref{eq: M impower sobolev} it is easy to see that Proposition~\ref{p: M frac transfer} extends to the range $[p_{\circ},p^{\circ}]$ and, using again \eqref{eq: M impower sobolev}, we can now repeat the arguments in Sections~\ref{s: M mm1} and \ref{s: M mm2}.
\end{proof}

\section{Two-parameter maximal operator}
Suppose that $A_{1},A_{2}\in\cA_{p}(\Omega)$. Define the associated two-parameter maximal operator by the rule 
$$
\gotM^{A_{1},A_{2}}f=\sup_{w,t>0}\mod{T^{A_{1}}_{w}T^{A_{2}}_{t}f}.
$$
\begin{theorem}\label{t: AppB tp}Let $p\in (1,\infty)$ and $A_{1},A_{2}\in\cA_{p}(\Omega)$. Then $\gotM^{A_{1},A_{2}}$ is bounded on $L^{p}(\Omega)$. 
\end{theorem}
\begin{proof}
In the special case when $A_{1}$ is real-valued, the associated semigroup is nonnegative and, therefore, we have  
\begin{equation}\label{eq: AppB domination}
\mod{T^{A_{1}}_{w}T^{A_{2}}_{t}f}\leq T^{A_{1}}_{w}\oM^{A_{2}}f\leq \oM^{A_{1}}\oM^{A_{2}}f.
\end{equation}
It follows from Theorem~\ref{t: M princ} that $\gotM^{A_{1},A_{2}}$ is bounded on $L^{p}(\Omega)$.

When $A_{1}$ is not real-valued, the domination in \eqref{eq: AppB domination} does not necessarily hold and we need another argument for proving the boundedness of $\gotM^{A_{1},A_{2}}$. It is a variant of the proof of Theorem~\ref{t: M princ}. 

Fix $f\in L^{p}(\Omega)$ and write
$$
T^{A_{1}}_{w}T^{A_{2}}_{t}f=T^{A_{1}}_{t+w}f+T^{A_{1}}_{w}(T^{A_{2}}_{t}f-T^{A_{1}}_{t}f).
$$
We clearly have 
$$
\gotM^{A_{1},A_{2}}f\leq \oM^{A_{1}}f+\sup_{w,t>0}\mod{T^{A_{1}}_{w}(T^{A_{2}}_{t}f-T^{A_{1}}_{t}f)}
$$
By Theorem~\ref{t: M princ} the maximal operator $\oM^{A_{1}}$ is bounded in $L^{p}(\Omega)$. As for the second term in the right hand side, we use Duhamel's \eqref{eq: M Duhamel1} formula and write
$$
T^{A_{1}}_{w}(T^{A_{2}}_{t}f-T^{A_{1}}_{t}f)=\int^{t}_{0}T^{A_{1}}_{w+s}(L_{A_{1}}-L_{A_{2}})T^{A_{2}}_{t-s}f\wrt s=I_{w,t}(f)-J_{w,t}(f),
$$
where
$$
I_{t,w}(f)=\int^{t}_{0}L_{A_{1}}T^{A_{1}}_{w+s}T^{A_{2}}_{t-s}f\wrt s;\quad J_{t,w}(f)=\int^{t}_{0}T^{A_{1}}_{w+s}L_{A_{2}}T^{A_{2}}_{t-s}f\wrt s.
$$
We now proceed as we did in Section~\ref{s: M mm1}. We write below the details only for the first integral.

Let $\alpha=\alpha(p)\in (0,1/2)$ be as in Proposition~\ref{p: M frac transfer}. Then,
$$
I_{t,w}(f)=\int^{t}_{0}\psi_{1-\alpha}((w+s)L_{A_{1}})U^{A_{1},A_{2}}_{\alpha}\psi_{\alpha}((t-s)L_{A_{2}})f\wrt\mu^{w,t}_{\alpha}(s),
$$
where
$$
\wrt \mu^{w,t}_{\alpha}(s)=(w+s)^{\alpha-1}(t-s)^{-\alpha}\ca_{[0,t]}(s)\wrt s.
$$
By \eqref{eq: M psi bis} we have
$$
\mu^{w,t}_{\alpha}([0,t])\leq \mu^{t}_{\alpha}([0,t])=\frac{\pi}{\sin(\alpha\pi)},\quad \forall w,t>0.
$$
Therefore,
$$
\sup_{w,t>0}\mod{I_{t,w}(f)}\leq\frac{\pi}{\sin(\alpha\pi)}\sup_{s,t>0}\mod{\psi_{1-\alpha}(sL_{A_{1}})U^{A_{1},A_{2}}_{\alpha}\psi_{\alpha}(tL_{A_{2}})f}.
$$
By Cowling's subordination, the maximal operator on the right-hand side is dominated by the sublinear operator
$$
\frac{1}{\pi^{2}}\int_{\R} \int_{\R}\mod{\Gamma(1-\alpha-iu)}\mod{\Gamma(\alpha-iv)}\mod{L^{iu}_{A_{1}}U^{A_{1},A_{2}}_{\alpha}L^{iv}_{A_{2}}f}\wrt u\wrt v.
$$
It follows from Stirling's formula \eqref{eq: M Stirling} in combination with Propositions~\ref{p: M frac transfer} and \ref{p: M Mixed} that the sub-linear operator above is bounded in $L^{p}(\Omega)$.
\end{proof}
\begin{appendices}
\section{Proof of Kato's theorem (Lemma~\ref{l: M Kato61}) via square functions}
In this appendix we give a proof of Lemma~\ref{l: M Kato61} which somewhat differs from the original one by Kato \cite{Kato61}. 
We find it likely that this proof may already be known, yet we where not able to find references that would confirm this.

Let $H$ be a complex Hilbert space. Let $\gota_{1},\gota_{2}$ be two densely defined, closed and sectorial (thus continuous) sesquilinear forms on $H$ \cite{Kat, Ouhabook, Haase}. Denote by $\oL_{\gotma_{1}}$,  and $\oL_{\gotma_{2}}$, respectively, the two associated Kato-sectorial operators on $H$. Suppose further that 
$$
\Dom(\gota_{1})=\Dom(\gota_{2})=\cV\quad \textrm{and}\quad \Re\gota_{1}(h)\sim \Re\gota_{2}(h),
$$
uniformly in $h\in\cV$. This is equivalent to $\norm{h}{\gotma_{1}}\sim\norm{h}{\gotma_{2}}$ uniformly in $h\in\cV$, where
$$
\|h\|^{2}_{\gotma_{j}}=\|h\|^{2}_{H}+\Re \gota_{j}(h),\quad h\in H,\  j=1,2.
$$
Moreover, we assume that $\oL_{\gotma_{1}}$, and thus also $\oL_{\gotma_{2}}$, are injective.
\medskip

Kato's theorem on fractional powers \cite{Kato61}, which is a part of Lemma~\ref{l: M Kato61}, can be rephrased by saying that, for every $\beta\in [0,1/2)$ we have 
\begin{equation}\label{eq: M KATLION}
\cV_{\beta}:=\Dom(\oL^{\beta}_{\gotma_{1}})=\Dom(\oL^{\beta}_{\gotma_{2}})\quad \textrm{and}\quad \norm{\oL^{\beta}_{\gotma_{1}}h}{}\sim\norm{\oL^{\beta}_{\gotma_{2}}h}{},\quad \forall h\in\cV_{\beta}.
\end{equation}
Once we have this result, Lions's improvement, $\cV_{\beta}=\left[H,\cV\right]_{2\beta}$ for all $\beta\in (0,1/2)$, follows from the modern theory of complex interpolation of domains of operators (see, for example \cite[Theorem~11.6.1]{MaCaSa01}). Indeed, in \eqref{eq: M KATLION} fix $\gota_{1}=\gota$ and take $\gota_{2}=\gotb=(\gota_{1}+\gota^{*}_{1})/2$ (the real part of $\gota$). The operator $\oH$ associated with $\gotb$ is nonnegative and self adjoint. Therefore, $\Dom(\oH^{1/2})=\Dom(\gotb)=\cV$ and $\cV_{\alpha/2}=\Dom(\oH^{\alpha/2})=\left[H,\cV\right]_{\alpha}$ for all $\alpha\in [0,1]$, because $\oH$ has bounded imaginary powers in $H$.
\medskip

We now prove \eqref{eq: M KATLION} by means of McIntosh's square functions \cite{Mc}. Recall the following well-known result.
\begin{lemma}\label{M: lMc} Let $\gota$ be a closed and densely defined sectorial form on a complex Hilbert space $H$. Let $\oL_{\gotma}$ denote the associated operator. Then for every $\gamma>0$ the square function
$$
G_{\gotma,\gamma}(x):=\left(\int^{\infty}_{0}\big\|\psi_{\gamma}(t\oL_{\gotma})x\big\|^{2}\frac{\wrt t}{t}\right)^{1/2},\quad x\in H
$$
is bounded on $H$.
\end{lemma}
\begin{proof}
By \cite{CrDe} the operator $\oL_{\gotma}$ has a bounded $H^{\infty}$-calculus on $H$, thus the result follows from \cite[Section~8]{Mc}. See also \cite{CDMY} and \cite[Corollary 7.1.17]{Haase}.
\end{proof}
Define $T^{\gotma_{j}}_{t}=\exp(-t\oL_{\gotma_{j}})$, $j=1,2$.
Fix $\beta\in (0,1/2)$, $x\in \Dom(\oL^{\beta}_{\gotma_{1}})$ and $y\in\Dom(\oL^{\beta}_{\gotma^{*}_{2}})$. We have
$$
\aligned
\mod{\sk{x}{\oL^{\beta}_{\gotma^{*}_{2}}y}}&=\mod{\int^{\infty}_{0}-\frac{\wrt}{\wrt t}\sk{T^{\gotma_{1}}_{t}x}{T^{\gotma^{*}_{2}}_{t}\oL^{\beta}_{\gotma^{*}_{2}}y}\wrt t}\\
&\leq \int^{\infty}_{0}\mod{\sk{\oL_{\gotma_{1}}T^{\gotma_{1}}_{t}x}{T^{\gotma^{*}_{2}}_{t}\oL^{\beta}_{\gotma^{*}_{2}}y}}\wrt t
+\int^{\infty}_{0}\mod{\sk{T^{\gotma_{1}}_{t}x}{\oL^{\beta+1}_{\gotma^{*}_{2}}T^{\gotma^{*}_{2}}_{t}y}}\wrt t\\
&=I+J.
\endaligned
$$

It follows from Lemma~\ref{M: lMc} that
$$
\aligned
I&=\int^{\infty}_{0}\mod{\sk{\psi_{1-\beta}(t\oL_{\gotma_{1}})\oL^{\beta}_{\gotma_{1}}x}{\psi_{\beta}(t\oL_{\gotma^{*}_{2}})y}}\frac{\wrt t}{t}\\&\leq \left(\int^{\infty}_{0}\big\|\psi_{1-\beta}(t\oL_{\gotma_{1}})\oL^{\beta}_{\gotma_{1}}x\big\|^{2}\frac{\wrt t}{t}\right)^{1/2} \left(\int^{\infty}_{0}\big\|\psi_{\beta}(t\oL_{\gotma^{*}_{2}})y\big\|^{2}\frac{\wrt t}{t}\right)^{1/2}\\
&\leqsim_{\beta}\norm{\oL^{\beta}_{\gotma_{1}}x}{}\norm{y}{}.
\endaligned
$$
As for $J$, we have
$$
\aligned
J&=\int^{\infty}_{0}\mod{\sk{\int^{\infty}_{t}\oL_{\gotma_{1}}T^{\gotma_{1}}_{s}x\wrt s}{\oL^{\beta+1}_{\gotma^{*}_{2}}T^{\gotma^{*}_{2}}_{t}y}\wrt s}\wrt t\\
&\leq\int^{\infty}_{1}\int^{\infty}_{0}t\mod{\sk{\oL_{\gotma_{1}}T^{\gotma_{1}}_{st}x}{\oL^{\beta+1}_{\gotma^{*}_{2}}T^{\gotma^{*}_{2}}_{t}y}}\wrt t\wrt s\\
&\leq\int^{\infty}_{1}s^{\beta-1}\int^{\infty}_{0}\mod{\sk{\psi_{1-\beta}(st\oL_{\gotma_{1}})\oL^{\beta}_{\gotma_{1}}x}{\oL_{\gotma^{*}_{2}}\psi_{\beta}(t\oL_{\gotma^{*}_{2}})y}}\wrt t\wrt s.
\endaligned
$$
By analyticity of the semigroup $(T^{\gotma_{1}}_{t})_{t>0}$ we have 
\begin{equation}\label{eq: zx M}
\psi_{1-\beta}(st\oL_{\gotma_{1}})\oL^{\beta}_{\gotma_{1}}x\in \Dom(\oL_{\gotma_{1}})\subseteq\Dom(\gota_{1})=\Dom(\gota_{2})=\Dom(\gota^{*}_{2}).
\end{equation}
It follows that
$$
\aligned
\sk{\psi_{1-\beta}(st\oL_{\gotma_{1}})\oL^{\beta}_{\gotma_{1}}x}{\oL_{\gotma^{*}_{2}}\psi_{\beta}(t\oL_{\gotma^{*}_{2}})y}&=\overline{\gota^{*}_{2}\left(\psi_{\beta}(t\oL_{\gotma^{*}_{2}})y,\psi_{1-\beta}(st\oL_{\gotma_{1}})\oL^{\beta}_{\gotma_{1}}x\right)}\\
&=\gota_{2}\left(\psi_{1-\beta}(st\oL_{\gotma_{1}})\oL^{\beta}_{\gotma_{1}}x,\psi_{\beta}(t\oL_{\gotma^{*}_{2}})y\right).
\endaligned
$$
Since $\gota_{2}$ is sectorial, there exists $C>0$, that does not depend on $x,y$ and $\beta$, such that
$$
\mod{\gota_{2}\left(\psi_{1-\beta}(st\oL_{\gotma_{1}})\oL^{\beta}_{\gotma_{1}}x,\psi_{\beta}(t\oL_{\gotma^{*}_{2}})y\right)}
\leq 
C
\sqrt{\Re\gota_{2}\left(\psi_{1-\beta}(st\oL_{\gotma_{1}})\oL^{\beta}_{\gotma_{1}}x\right)}
\sqrt{\Re\gota_{2}\left(\psi_{\beta}(t\oL_{\gotma^{*}_{2}})y\right)}.
$$
Recall that we are assuming that $\Re\gota_{2}(h)\sim\Re\gota_{1}(h)$ and note that $\Re\gota^{*}_{2}(h)=\Re\gota_{2}(h)$. Hence, by Cauchy-Schwartz inequality,
$$
J\leqsim 
\int^{\infty}_{1}s^{\beta-1}\left(\int^{\infty}_{0}\Re\gota_{1}
\left(\psi_{1-\beta}(st\oL_{\gotma_{1}})\oL^{\beta}_{\gotma_{1}}x\right)\wrt t\right)^{1/2}
\left(\int^{\infty}_{0}\Re\gota^{*}_{2}\left(\psi_{\beta}(t\oL_{\gotma^{*}_{2}})y\right)\wrt t\right)^{1/2}\wrt s.
$$ 
Then a change of variable in the first integral with respect to the variable $t$ gives
$$
J\leqsim 
\frac{1}{1/2-\beta}
\left(\int^{\infty}_{0}\Re\gota_{1}\left(\psi_{1-\beta}(t\oL_{\gotma_{1}})\oL^{\beta}_{\gotma_{1}}x\right)\wrt t\right)^{1/2}\left(\int^{\infty}_{0}\Re\gota^{*}_{2}\left(\psi_{\beta}(t\oL_{\gotma^{*}_{2}})y\right)\wrt t\right)^{1/2}.
$$
We now observe that, by \eqref{eq: zx M},
$$
\aligned
&\gota_{1}\left(\psi_{1-\beta}(t\oL_{\gotma_{1}})\oL^{\beta}_{\gotma_{1}}x\right)\wrt t=\sk{t\oL_{\gotma_{1}}\psi_{1-\beta}(t\oL_{\gotma_{1}})\oL^{\beta}_{\gotma_{1}}x}{\psi_{1-\beta}(t\oL_{\gotma_{1}})\oL^{\beta}_{\gotma_{1}}x}\frac{\wrt t}{t}\\
&\gota^{*}_{2}\left(\psi_{\beta}(t\oL_{\gotma^{*}_{2}})y\right)\wrt t=\sk{t\oL_{\gotma^{*}_{2}}\psi_{\beta}(t\oL_{\gotma^{*}_{2}})y}{\psi_{\beta}(t\oL_{\gotma^{*}_{2}})y}\frac{\wrt t}{t}.
\endaligned
$$
We now use Cauchy-Schwartz inequality and Lemma~\ref{M: lMc}. It follows that
$$
J\leqsim\frac{1}{1/2-\beta}\norm{\oL^{\beta}_{\gotma_{1}}x}{}\norm{y}{}
$$
and combining the estimate above with the similar estimate for $I$, we conclude that
$$
\mod{\sk{x}{\oL^{\beta}_{\gotma^{*}_{2}}y}}\leqsim_{\beta}\norm{\oL^{\beta}_{\gotma_{1}}x}{}\norm{y}{},\quad \forall x\in\Dom(\oL^{\beta}_{\gotma_{1}}),\quad y\in\Dom(\oL^{\beta}_{\gotma^{*}_{2}}),\quad \forall\beta\in (0,1/2)
$$
which is equivalent to
$$
\Dom(\oL^{\beta}_{\gotma_{1}})\subseteq\Dom(\oL^{\beta}_{\gotma_{2}})\quad \textrm{and}\quad \norm{\oL^{\beta}_{\gotma_{2}}h}{}\leqsim_{\beta}\norm{\oL^{\beta}_{\gotma_{1}}h}{}\quad \forall h\in \Dom(\oL^{\beta}_{\gotma_{1}}).
$$
The reverse inclusion and inequality (and thus \eqref{eq: M KATLION}) follow by reversing the role of $\gota_{1}$ and $\gota_{2}$ in the previous proof. 

\end{appendices}

\subsection*{Acknowledgements}

A. Carbonaro was partially supported by the ``National Group for Mathematical Analysis, Probability and their Applications'' (GNAMPA-INdAM).
O. Dragi\v{c}evi\'c was partially supported by the Slovenian Research Agency, ARRS (research grant J1-1690 and research program P1-0291).

\bibliographystyle{amsplain}
\bibliography{biblio_Maximal}

\end{document}